\documentclass[12pt]{article} 

\usepackage{amsmath}
\usepackage{amsthm}
\usepackage{amsfonts}
\usepackage{mathrsfs}
\usepackage{stmaryrd}
\usepackage{setspace}
\usepackage{fullpage}
\usepackage{amssymb}
\usepackage{breqn}
\usepackage{enumitem}
\usepackage{bbold} 
\usepackage{authblk}
\usepackage{comment}
\usepackage{hyperref}
\usepackage{pgf,tikz}
\usepackage{graphicx}
\usepackage{subcaption}

\newtheorem{thm}{Theorem}[section]

\newtheorem{prop}[thm]{Proposition}

\newtheorem{observation}[thm]{Observation}
\newtheorem*{theorem*}{Theorem}
\newtheorem*{prop*}{Proposition}

\newcommand\ex{\ensuremath{\mathrm{ex}}}

\newcommand\cF{{\mathcal F}}

\newcommand\cH{{\mathcal H}}

\newcommand\cN{{\mathcal N}}

\usepackage{lineno}

\newcommand{\ignore}[1]{}

\title{On generalized Turán problems with bounded matching number}
\author{D\'aniel Gerbner\footnote{Alfr\'ed R\'enyi Institute of Mathematics, E-mail: \texttt{gerbner@renyi.hu.}}}
\date{}

\begin{document}

\maketitle

\begin{abstract}
Given a graph $H$ and a family of graphs $\cF$, the generalized Tur\'an number $\ex(n,H,\cF)$ is the maximum number of copies of $H$ in an $n$-vertex graphs that do not contain any member of $\cF$ as a subgraph. Recently there has been interest in studying the case $\cF=\{F,M_{s+1}\}$ for arbitrary $F$ and $H=K_r$. We extend these investigations to the case $H$ is arbitrary as well.
\end{abstract}

\section{Introduction}

Given a positive integer $n$ and a graph $F$, the largest number of edges in an $n$-vertex graph that does not contain $F$ as a subgraph is denoted by $\ex(n,F)$ and is called the \textit{Tur\'an number} of $F$. If every graph from a family $\cF$ of graphs is forbidden, then we use the notation $\ex(n,\cF)$. 

The seminal result of this area is due to Tur\'an \cite{T}, who showed that $\ex(n,K_{r+1})=|E(T_r(n))|$, where $T_r(n)$ denotes the Tur\'an graph, i.e., the complete $r$-partite graph with each part of order $\lfloor n/r\rfloor$ or $\lceil n/r\rceil$. Another early result is due to Erd\H os and Gallai \cite{eg}, who showed that for the matching $M_{s+1}$ consisting of $s+1$ independent edges, we have $\ex(n,M_{s+1})=\max \{\binom{2s+1}{2}, \binom{s}{2}+s(n-s)\}$.

The above problems were combined by Alon and Frankl \cite{af}, who showed that for $n\ge 2s+1$ we have $\ex(n,\{K_{r+1},M_{s+1}\})=\max \{\binom{2s+1}{2}, |E(T_r(s))|+s(n-s)\}$. Gerbner \cite{ger} considered $\ex(n,\{F,M_{s+1}\})$ in general, and determined its value apart from a constant additive term. Here and everywhere in this paper $n$ goes to infinity, while other parameters are considered constant.

Let $\cH(F)$ denote the family of graphs obtained by deleting an independent set from $F$
and $\cH^*(F)$ denote the family of graphs obtained by deleting an independent set of order $\alpha(F)$ from $F$. 
If $F$ is bipartite, then let $p=p(F)$ denote the smallest possible order of a color class in a proper two-coloring of $F$.

\begin{thm}[Gerbner \cite{ger}]\label{k2}
   \textbf{(i)} If $\chi(F)>2$ and $n$ is large enough, then $\ex(n,\{F,M_{s+1}\})=\ex(s,\cH(F))+s(n-s)=s(n-s)+O(1)$.

   \textbf{(ii)} Let $\chi(F)=2$ and $n$ large enough. If $p(F)>s$ then $\ex(n,\{F,M_{s+1}\})=\ex(n,M_{s+1})=|E(K_s+\overline{K_{n-s}})|=s(n-s)+\binom{s}{2}$. If $p(F)\le s$, then $\ex(n,\{F,M_{s+1}\})=(p(F)-1)n+O(1)$.
\end{thm}

Given a positive integer $n$ and two graphs $H$ and $F$, the largest number of copies of $H$ in $n$-vertex $F$-free graphs is denoted by $\ex(n,H,F)$. If every graph from a family $\cF$ is forbidden, we use the notation $\ex(n,H,\cF)$. After several sporadic results (see e.g. \cite{gypl,zykov}), the systematic study of these functions was initiated by Alon and Shikhelman \cite{ALS2016}.

Let $\cN(H,G)$ denote the number of copies of $H$ in $G$.
Zykov \cite{zykov} showed that $\ex(n,K_k,K_{r+1})=\cN(K_k,T_r(n))$.
Wang \cite{wang} determined $\ex(n,K_k,M_{s+1})$. 

Ma and Hou \cite{mh} combined the above problems and studied $\ex(n,K_k,\{F,M_{s+1}\})$. They showed the following. 

\begin{thm}[Ma and Hou \cite{mh}]\label{kk}

       \textbf{(i)} If $\chi(F)>2$, then $\ex(n,K_k,\{F,M_{s+1}\})=\ex(s,K_{k-1},\cH(F))n+O(1)$.

   \textbf{(ii)} Let $\chi(F)=2$ and $n$ be large enough. If $p(F)>s$ then $\ex(n,K_k,\{F,M_{s+1}\})=\ex(n,K_k,M_{s+1})=\cN(K_k,K_s+\overline{K_{n-s}})$. If $p(F)\le s$, then $\ex(n,K_k\{F,M_{s+1}\})=\binom{p(F)-1}{k-1}n+O(1)$.
\end{thm}

A small improvement was obtained in \cite{zc}. One can say that the above theorem determines $\ex(n,K_k,\{F,M_{s+1}\})$ apart from an additive error term $O(1)$. Note however that the main term depends on $\ex(s,K_{r-1},\cH(F))$, which may be unknown (unlike in Theorem \ref{k2}).
Nevertheless, we will use the phrase ``determine'' if the bound only depends on the solution of a generalized Tur\'an problem on $O(1)$ vertices.

In this paper we will study $\ex(n,H,\{F,M_{s+1}\})$ for arbitrary $F$ and $H$. Before turning to that problem, let us discuss what we know about $\ex(n,H,M_{s+1})$ in general.

The order of magnitude of $\ex(n,H,M_{s+1})$ was determined in \cite{ger6} for every graph $H$. Given a graph $G$ and a set $U\subset V(G)$, the \textit{partial $(m,U)$-blowup of $G$} is obtained the following way. We replace each vertex $u\in U$ with $m$ vertices $u_1,\dots,u_m$ and replace each vertex $u\not\in U$ with one vertex $u_1$. Then we replace each edge $uv$ by the edges $u_iv_j$ for each $i,j$, thus we replace each edge by 1, $m$ or $m^2$ edges. Let us consider a largest set $U\subset V(H)$ such that no partial $(m,U)$-blowup of $H$ contains $M_{s+1}$, and let $b(H)=b(H,s)$ denote the order of $U$. A theorem in \cite{ger6} shows that $\ex(n,H,M_{s+1})=\Theta(n^{b(H)})$.

It is easy to see that in our case $b(H)\le\alpha(H)$ (recall that $\alpha(H)$ is the order of a largest independent set in $H$). Indeed, if we blow up both endpoints of an edge, then the resulting graph contains $M_m$. 

\smallskip

We state and prove our results in Section \ref{ketto}. Our main results are the following. We determine the order of magnitude of $\ex(n,H,\{F,M_{s+1}\})$ for every $H$, $F$ and $s$. For most graphs $F$ (including each non-bipartite graph), we determine $\ex(n,H,\{F,M_{s+1}\})$ apart from an additive error term $O(n^{\alpha(H)-1})$ for every $H$ and $s$. We answer a question of Ma and Hou \cite{mh} by showing that in the case $\chi(F)>2$, the $O(1)$ error term in Theorem \ref{kk} may be necessary even if $\ex(s,K_{k-1},\cH(F))>0$. Moreover, we provide a characterization in our more general setting, i.e., we show when the error term $O(n^{\alpha(H)-1})$ can be omitted from our result. 

Ma and Hou \cite{mh} determined $\ex(n,K_k,\{K_{r+1},M_{s+1}\}$ for each value of $n,k,r,s$. The extremal graph is either $T_r(2s+1)$ or obtained from $T_{r-1}(s)$ by adding $n-s$ vertices, each joined to each of the $s$ old vertices. We extend their result to counting any complete multipartite graph with at most $r$ parts. The extremal graphs are the same, except the Tur\'an graph is replaced by a not necessarily balanced complete $r$-partite graph in both cases.

\section{Results and proofs}\label{ketto}

We will use the following theorem of Berge and Tutte \cite{berg}.

\begin{thm}[Berge-Tutte]  A graph $G$ is $M_{s+1}$-free if and only if there is a set $B\subset V(G)$ such that removing $B$ cuts $G$ to connected components $G_1,\dots,G_m$ with each $A_i=V(G_i)$ of odd order such that $|B|+\sum_{i=1}^m \frac{|A_i|-1}{2}\le s$.
\end{thm}

We will use the following simple facts about the partition in the Berge-Tutte theorem.

\begin{observation}\label{obsi}
    Let $H$ be an arbitrary graph and $G$ be an $n$-vertex $M_{s+1}$-free graph. Let $B$ be a subset of $V(G)$ guaranteed by the Berge-Tutte theorem. Then

    \textbf{(i)} Any vertex $v$ outside $B$ is contained in $O(n^{\alpha(H)-1})$ copies of $H$.

    \textbf{(ii)} There are $O(n^{\alpha(H)-1})$ copies of $H$ containing at least one edge outside $B$.
\end{observation}

\begin{proof} To show \textbf{(i)}, observe that the vertices of $G$ form an independent set with $v$ except for $O(1)$ vertices, thus in a copy of $H$ containing $v$ there are $O(1)$ ways to pick each other vertex, except at most $\alpha(H)-1$ vertices that form an independent set with $v$ in $H$.
To show \textbf{(ii)}, observe that there are $O(1)$ edges outside $B$. 
    There are $O(1)$ ways to pick such an edge $uv$, and then there are $O(n^{\alpha(H)-1})$ copies of $H$ that contain $v$.
\end{proof}

We will use the following simple statement.

\begin{prop}\label{matc} If $|V(H)|\le s+\alpha(H)$, then
    $\ex(n,H,M_{s+1})=\cN(H,K_s+\overline{K_{n-s}})$ for sufficiently large $n$. Otherwise $\ex(n,H,M_{s+1})=O(n^{\alpha(H)-1})$.
\end{prop}

\begin{proof}
Assume first that $|V(H)|\le s+\alpha(H)$, let $G$ be an extremal graph for $\ex(n,H,M_{s+1})$ and apply the Berge-Tutte theorem to obtain $B$. If $|B|=s$, then $G$ is a subgraph of $K_s+\overline{K_{n-s}}$. If $|B|<s$, then we compare $\cN(H,G)$ to $\cN(H,K_s+\overline{K_{n-s}})$. Let $H_0$ denote a graph we obtain from $H$ by deleting $\alpha(H)$ independent vertices. Then $H_0$ has at most $s$ vertices and clearly there are less copies of $H$ inside $B$ than inside $K_s$. As we can extend $H_0$ to $H$ using the vertices outside $B$ $\Theta(n^{\alpha(H)})$ ways, we lose $\Theta(n^{\alpha(H)})$ copies of $H$. On the other hand, we only gain new copies of $H$ that contain vertices in $A_i$ with $|A_i|>1$. There are $O(1)$ such vertices and each of them is contained in $O(n^{\alpha(H)-1})$ copies of $H$
by Observation \ref{obsi}.   

Assume now that $|V(H)|> s+\alpha(H)$ and consider a partial $(m,U)$-blowup of $H$ that does not contain $M_{s+1}$, with $m>s$. Then $U$ is an independent set. If $|U|=\alpha(H)$, then each of the at least $s+1$ vertices $v_1,\dots,v_{s+1}$ of $H$ outside $U$ have a neighbor in $U$. We pick an arbitrary such neighbor $u$ for $v_i$. In the partial blow-up, $u$ is replaced by $m$ vertices. If $u$ is picked for multiple vertices $v_i$, we pick a distinct one of the replacing vertices for each $v_i$ and denote it by $u_i$. Then the edges $u_iv_i$ form an $M_{s+1}$, a contradiction.
\end{proof}

Recall that $b(H)=b(H,s)$ denotes the largest order of a set such that $M_{s+1}$ is not a subgraph of any partial $(m,U)$-blowup of $H$.
Let $b'(H)=b'(H,F,s)$ denote the largest order of a set $U\subset V(H)$ such that neither $F$ nor $M_{s+1}$ is a subgraph of any partial $(m,U)$-blowup of $H$. 

\begin{thm}\label{ujj2}
   For every $H$, $F$ and $s$ we have $\ex(n,H,\{F,M_{s+1}\})=\Theta(n^{b'(H)})$.
\end{thm}

\begin{proof}
The lower bound is given by the partial blowup in the definition of $b'$.

  Let $G$ be an $n$-vertex $\{F,M_{s+1}\}$-free graph. 
  We apply the Berge-Tutte theorem to $G$, to obtain a set $B$. Let $A$ be the union of components $A_i$ of order more than 1. Observe that $|A\cup B|\le 2s$. Assume that $G$ contains at least $cn^{b'(H)}$ copies of $H$, where $c$ is sufficiently large. It is more convenient for us to talk about embeddings of $H$: each copy of $H$ is given by an injective function $f:V(H)\rightarrow V(G)$ such that for each edge $uv$ of $H$, $f(u)f(v)$ is an edge of $G$.
  
  We will choose an edge-covering set $S$ of vertices in $H$, i.e., a set such that each edge of $H$ is incident to at least one of its vertices. Observe that the vertex set of each copy of $H$ in $G$ intersects $A\cup B$ in an edge-covering set. There are at most $2^{|V(H)|}$ choices to pick $S$; we pick one that appears as the intersection of $A\cup B$ with at least $cn^{b'(H)}/2^{|V(H)|}$ copies of $H$.

From now on we deal with copies (embeddings) of $H$ where the vertices of $S$ are in $A\cup B$, now we fix those vertices. There are $\binom{|A\cup B|}{|S|}|S|!\le 2^{2s}|V(H)|!$ ways to embed the vertices of $H$ into $A\cup B$; we fix the embedding that can be extended to at least $cn^{b'(H)}/2^{|V(H)|+2s}|V(H)|!$ embeddings of $H$. It is left to embed the remaining vertices, which form an independent set in $H$. For each vertex $v$ of $H$ outside $S$, the neighborhood of $v$ is a subset $S(v)$ of $S$, thus we have to pick from $G$ a neighbor of the corresponding vertices in $A\cup B$. Let $S(v)=\{u_1,\dots,u_k\}$, then we are interested in the order of the common neighborhood of $f(u_1),\dots,f(u_k)$ in $G$, which we denote by $x(v)$.

Recall that the partial $(m,U')$-blowup of $H$ contains $F$ or $M_{s+1}$ for each set $U'$ of vertices of $H$ with order $b'(H)+1$, for some $m$. For each $U'$, we pick the smallest possible $m$. As there are $O(1)$ ways to choose $U'$, we can pick the largest of those numbers $m$, let us denote it by $m'$.

Let $q$ denote the number of vertices $v$ in $V(H)\setminus S$ such that $x(v)\ge (b'(H)+1) m'+|V(H)|$. If $q\le b'(H)$, then we continue by picking the images of the vertices outside $S$. There are at most $n$ ways to pick $q$ of the vertices and less than $(b'(H)+1) m'+|V(H)|$ ways to pick the other vertices. Therefore, there are at most $((b'(H)+1) m'+|V(H)|)n^{b'(H)}$ embeddings after fixing the vertices of $S$, contradicting the lower bound $cn^{b'(H)}/2^{|V(H)|+2s}|V(H)|!$ established earlier, if $c$ is sufficiently large.

We obtained that $q\ge b'(H)+1$. First we pick an embedding of $H$ that extends the earlier embedding of $S$. We fix the embedding of each of the vertices except for $b'(H)+1$ vertices $v_1,\dots,v_{b'(H)+1}$ with $x(v_i)\ge (b'(H)+1) m'+|V(H)|$. We go through the vertices $v_i$ one by one, and greedily pick $m'$ vertices outside $A\cup B$ that they can be embedded to, such that each vertex is picked only once. This is doable, since each time we have to avoid the at most $|V(H)|$ vertices already fixed and the at most $b'(H)m'$ vertices picked earlier. This means we have a way to fix all but $b'(H)+1$ vertices and there are at least $m'$ common neighbors of $S(v_i)$ for the unfixed vertices $v_i$. In other words, we have in $G$ a partial $(m',U)$-blowup with $U=\{v_1,\dots,v_{b'(H)+1}$, which contains $F$ or $M_{s+1}$, a contradiction.
\end{proof}

 We are able to generalize Theorem \ref{kk} for most instances by determining $\ex(n,H,\{F,M_{s+1}\})$ apart from an additive term $O(n^{\alpha(H)-1})$.

\begin{prop}\label{main} \textbf{(i)} Let $\chi(F)>2$.
    There is an $s$-vertex $\cH(F)$-free graph $G_0$ such that $\ex(n,H,\{F,M_{s+1}\})=\cN(H,G_0+\overline{K_{n-s}})+O(n^{\alpha(H)-1})$.

    \textbf{(ii)}  Let $\chi(F)=2$. If $p(F)>s$ then $\ex(n,H,\{F,M_{s+1}\})=\ex(n,H,M_{s+1})=\cN(H,K_s+\overline{K_{n-s}})+O(n^{\alpha(H)-1})$. 
\end{prop}

We remark that $\cN(H,G_0+\overline{K_{n-s}})$ can be determined by determining $\cN(H',G_0)$ for every $H'\in \cH(H)$ (and for the main term, it is enough to deal with $H'\in \cH^*(H)$). Therefore, we want to count multiple graphs in $s$-vertex $\cH(F)$-free graphs $G_0$. Such problems were studied in \cite{ger7}. We also have weights here: each copy of $H'$ is counted $\binom{n-s}{|V(H)|-|V(H')|}$ times, but we are only interested in the main term because of the error term $O(n^{\alpha(H)-1})$. Still, this problem seems to be complicated in general. 

We can deal with counting multiple graphs at the same time if the same graph is extremal for each of them, in particular if the same graph is extremal for each subgraphs of $H$. Two examples of this are counting linear forests if $F$ has a color-critical edge \cite{ger5}, and counting any graph $H$ if $F$ has a color-critical edge and chromatic number at least $300|V(H)|^9+1$ \cite{mnnrw,dahi}. 

\begin{proof}
    Let $G$ be an extremal graph for $\ex(n,H,\{F,M_{s+1}\})$ and apply the Berge-Tutte theorem to obtain $B$. Let $G_0$ be the graph we obtain from $G[B]$ by adding $s-|B|$ isolated vertices. 
    
    
    Consider different types of copies of $H$. 
    There are copies of $H$ that intersect $V(G)\setminus B$ in an independent set. There are clearly at most $\cN(H,G_0+\overline{K_{n-|B|}})$ such copies. The $s-|B|$ additional isolated vertices in $G_0+\overline{K_{n-|B|}}$ compared to $G_0+\overline{K_{n-s}}$ are in $O(n^{\alpha(H)-1})$ copies of $H$ by Observation \ref{obsi}. 
Also by Observation \ref{obsi} there are $O(n^{\alpha(H)-1})$ copies of $H$ 
that contain an edge outside $B$. This completes the proof of \textbf{(i)}.

Assume now that $F$ is bipartite and $p>s$, then $F$ contains $M_{s+1}$ and Proposition \ref{matc} implies the rest of the statement. 
\end{proof}

In the remaining case we have that $p(F)\le s$. It is easy to see using Observation \ref{obsi} that we can obtain a subgraph of $K_s+\overline{K_{n-s}}$ from the extremal graph by deleting some edges such that we delete only $O(n^{\alpha(H)-1})$ copies of $H$. In other words, for our purposes it is enough to deal with subgraphs of $K_s+\overline{K_{n-s}}$ and then we can forget about forbidding $M_{s+1}$. However, it is not clear what subgraph of the part $K_s$ is allowed. For example, if $F=K_{s,t}$ with $s\le t$, then we can have the full $K_s$ there. In this case we still have to solve the following unbalanced bipartite version of the generalized Tur\'an problem (with an additive error term $O(n^{\alpha(H)-1})$): what is the largest number of copies $H$ in $F$-free subgraphs of $K_{s,n-s}$? Moreover, we also have to solve this for counting subgraphs of $H$.

\smallskip

A natural goal is to get rid of the $O(n^{\alpha(H)-1})$ term in Proposition \ref{main}.
Clearly, if $\ex(s,\cH(H),\cH(F))=0$, then the $O(n^{\alpha(H)-1})$ term can be omitted if and only if $\ex(n,H,F)=0$, i.e., $H$ contains $F$. In the case $H=K_r$, $\chi(F)>2$ and $\ex(s,K_{r-1},\cH(F))>0$, Ma and Hou \cite{mh} asked whether one can get rid of the $O(1)$ term. We answer this negatively. Given $s$, $\cH(H)$ and $\cH(F)$, and an $s$-vertex $\cH(F)$-free graph $G_0$, we say a vertex $v$ of $G_o$ is \textit{useless} if it is not contained in any copy of any graph from $\cH(H)$. Even with one forbidden graph, there are examples with useless vertices. The simplest example is when $\ex(s,\cH(H),\cH(F))=O(1)$ and $s$ is sufficiently large, for example $\ex(s,K_3,M_2)\le 1$, thus there is a useless vertex if $s\ge 4$. See \cite{germet} for more on generalized Tur\'an problems with constant values. Another example is $\ex(s,K_r,S_k)$, where the extremal graph consists of $\lfloor s/k\rfloor$ copies of $K_k$ and a clique on the remaining vertices \cite{chase}. If $s-k\lfloor s/k\rfloor<r$, then those additional vertices are useless.

We characterize when we can completely get rid of the error term.
Observe that there is a useless vertex in an extremal graph if and only if $ex(s,\cH(H),\cH(F))=\ex(s-1,\cH(H),\cH(F))$. 

\begin{prop}\label{usi} Let $\chi(F)>2$, $s\ge 2$ and $n$ be sufficiently large. There is an $s$-vertex $\cH(F)$-free graph $G_0$ such that $\ex(n,H,\{F,M_{s+1}\})=\cN(H,G_0+\overline{K_{n-s}})$ if and only if $ex(s,\cH(H),\cH(F))>\ex(s-1,\cH(H),\cH(F))$.
\end{prop}

\begin{proof}
We follow the proof of Proposition \ref{main} (\textbf{(i)}. The additional error term $O(n^{\alpha(H)-1})$ came from two parts of the proofs. If $|B|<s$, then the proof gives the upper bound $\cN(H,G)\le\ex(s-1,\cH(H),\cH(F))\binom{n-|B|}{\alpha(H)}+O(n^{\alpha(H)-1})$. This is clearly less than $\ex(s,\cH(H),\cH(F))\binom{n-s}{\alpha(H)}=\cN(H,G_0+\overline{K_{n-s}})$ for $n$ sufficiently large, a contradiction. 

If $|B|=s$, then we cannot have any edge outside $B$, which was the other possible reason for having the error term.
\end{proof}

It is a natural question when we have useless vertices, but we are unable to answer this question in general. Given a family $\cH$ of graph, let $\chi(\cH)$ denote the smallest chromatic number among graphs in $\cH$. We can show that there are no useless vertices if $\chi(H)<\chi(F)$ and $n$ is large enough. We also state the result in the form that we use here.

\begin{prop}\label{usel} 
\textbf{(i)} If $\chi(H)<\chi(F)$ and $n$ is sufficiently large, then $\ex(n,H,F)>\ex(n-1,H,F)$.

\textbf{(ii)}  If $\chi(\cH^*(H))<\chi(\cH(F))$ and $n$ is sufficiently large, then $\ex(n,\cH(H),\cH(F))=\ex(n-1,\cH(H),\cH(F))$. 
\end{prop}

\begin{proof} Let us start by proving \textbf{(i)}.
   Let $G$ be an extremal graph for $\ex(n,H,F)$ and assume indirectly that $v$ is a useless vertex. Let $\chi(H)=k$. Let $H'$ denote the $|V(F)|$-blowup of $H$. If $G$ does not contain $H'$, then 
 $\cN(H,G)\le \ex(n,H,H')=o(n^{|V(H)|})$. Indeed, Alon and Shikhelman \cite{ALS2016} showed that $\ex(n,H,H')=\Theta(n^{|V(H)|})$ if and only if $H'$ is not contained in any blow-up of $H$. Clearly we have $\ex(H,F)\ge \cN(H,T_k(n))=\Theta(n^{|V(H)|})$, thus for $n$ sufficiently large we obtain a contradiction.

 Consider a copy of $H'$ in $G$, and let 
 $U$ be one of the blown-up classes in $H'$. We delete all the edges incident to $v$ from $G$ and add all the edges from $v$ to the common neighbors of vertices of $U$ to obtain a graph $G'$. Obviously, $v$ is connected to some of the other blown-up classes in $H'$, and in particular there are some new copies of $H$, thus $\cN(H,G')>\cN(H,G)=\ex(n,H,F)$. However, we claim that $G'$ is $F$-free, which clearly gives a contradiction. Indeed, assume that there is a copy of $F$ in $G$ that is not in $G$. Then this copy must contain $v$ and there is a vertex $u\in U$ that is not contained in this copy. Then we can replace $v$ by $u$ to obtain another copy of $F$, since the neighbors of $v$ are all neighbors of $u$. But this new copy of $F$ is in $G$, a contradiction.

 The proof of \textbf{(ii)} goes similarly. Let $\cH'$ denote the family of graphs that are $|V(F)|$-blowups of some graphs in $\cH^*(H)$. If $G$ is $\cH'$-free, then the same argument as above completes the proof. If $G$ contains the $|V(F)|$-blowup  $H_0'$ for some $H_0\in\cH^*(H)$, then let $U_0$ be one of the blown-up classes. Similarly to the proof of \textbf{(i)}, we delete all the edges incident to $v$ from $G$ and add all the edges from $v$ to the common neighbors of vertices of $U_0$ to obtain a graph $G''$. Clearly the number of copies of graphs in $\cH^*(H)$ increases. If there is a member of $\cH(F)$ in $G''$, then it contains $v$ and we can replace $v$ by an unused vertex of $U_0$, a contradiction completing the proof.
\end{proof}

Recall that Ma and Hou \cite{mh} showed the following.

\begin{thm}[Ma and Hou \cite{mh}]\label{klikks} Let $n\ge 2s+1$ and $r\ge k\ge 3$. Then $\ex(n,K_k,\{K_{r+1},M_{s+1}\})=\max\{\cN(K_k,T_{r-1}(s)+\overline{K_{n-s}}),\cN(K_k,T_r(2s+1))\}$.\end{thm}

We determine $\ex(n,H,\{K_{r+1},M_{s+1}\})=\cN(H,K)$ where $H$ is a complete multipartite graph with at most $r$ parts. Note that it is known \cite{gypl} that $\ex(n,H,K_{r+1})=\cN(H,T)$, where $T$ is a complete $r$-partite graph. However, it is a complicated calculation to find which complete $r$-partite graph contains the most copies of $H$ that we cannot handle in general. We have the same problem in our theorem below.


The proof of Theorem \ref{klikks} in \cite{mh} uses Zykov symmetrization. Given a $K_{r+1}$-free graph $G$ and two non-adjacent vertices $u$ and $v$, we say that we \textit{symmetrize} $u$ to $v$ if we delete all the edges incident to $u$, and for each edge $vw$, we add the edges $uw$. In other words, we replace the neighborhood of $u$ by the neighborhood of $v$. It is easy to see that such symmetrization steps do not create $K_{r+1}$, and either symmetrizing $u$ to $v$ or symmetrizing $v$ to $u$ does not decrease the number of copies of $K_k$. We apply such steps as long as we can. One has to show that this process terminates. Afterwards, it is clear that the resulting graph is complete multipartite. This argument was used by Zykov \cite{zykov} to show that the Tur\'an graph contains the most copies of $K_k$ among $n$-vertex $K_{r+1}$-free graphs. 

Let $H$ be a complete $k$-partite graph and $k\le r$.
It was observed in \cite{gypl} that the Zykov symmetrization does not decrease the number of copies of $H$ either. They used is to show that a complete $r$-partite graph contains the most copies of $H$ among $n$-vertex $K_{r+1}$-free graphs. Analogously, we can use this to extend the proof of Theorem \ref{klikks} in \cite{mh} to our more general version. A part of the proof from \cite{mh} extends without further ideas. We state the generalization of that part in the next proposition and
 we only give a sketch of the proof. The interested Reader may find more details in \cite{mh}.

\begin{prop}\label{propi} Let $H$ be a complete multipartite graphwith at most $r$ parts. Then there is an extremal graph $G$ for $\ex(n,H,\{K_{r+1},M_{s+1}\})$ such that applying the Berge-Tutte theorem, 
$G[B]$ is a complete multipartite graph and
$|A_2|=\dots =|A_m|=1$.
\end{prop}

\begin{proof}[Sketch of proof]
First we apply Zykov symmetrization for $u,v$ both inside $B$. This does not create $K_{r+1}$ because of the properties of the symmetrization mentioned earlier. This also does not create $M_{s+1}$, because in the resulting graph we still have the same $B$ that cuts the graph into parts as prescribed in the Berge-Tutte theorem, the same parts as originally. We have to show that the process terminates. We start with symmetrizing vertices with different neighborhood outside $B$. This way we obtain that vertices inside $B$ have the same neighborhood outside $B$, since the number of distinct neighborhoods outside $B$ decreases. After that we symmetrize inside $B$. Whenever we symmetrize $u$ with neighborhood $A$ to $v$, in the next steps we symmetrize each vertex with neighborhood $A$ to $v$. After all these steps, the number of different neighborhoods decreases, thus this process terminates.

If there are two components in $G-B$, say $A_1$ and $A_2$ that contain edges, then without loss of generality $v_1\in A_1$ is in the largest number of copies of $K$. Let $v_2,v_3\in A_2$ such that $A_2\setminus \{v_2,v_3\}$ is connected, and symmetrize $v_2$ and $v_3$ both to $v_1$. Then we do not create $M_{s+1}$, since the same $B$ cuts the graph into parts as prescribed in the Berge-Tutte theorem. Note that the number of vertices in $A_1$ and $A_2$ changed, but neither the parity of $|A_1|$ and $|A_2|$, nor $\sum_{i=1}^m \frac{|A_i|-1}{2}$ changed.
\end{proof}

\begin{prop}\label{propp}
    If a complete $r$-partite $M_{s+1}$-free graph has more than $2s+1$ vertices, then a part has order at least $n-s$.
\end{prop}

\begin{proof}
    Let us pick edges one by one, such that we always pick an edge between the two largest part, and remove its endvertices. We are going to show that after picking $s$ edges, there are vertices in at least two parts. 
    As we have at least two vertices left, the only possible problem is that we are left with a part $B_0$ of order at least two. Observe that if two parts have the same order at one point of this algorithm, then their order will never differ by more than one. Therefore, $B_0$ was always the largest part and we removed a vertex from $B_0$ and a vertex outside $B_0$. There are at least $s+1$ vertices avoiding $B_0$ and in each step we removed one, thus there is an undeleted vertex outside $B_0$. This way we find an $M_{s+1}$, a contradiction completing the proof.
\end{proof}

Now we are ready to state and prove our generalization of Theorem \ref{klikks}.

\begin{thm}\label{multip}     If $H$ is a complete multipartite graph with at most $r$ parts, then $\ex(n,H,\{K_{r+1},M_{s+1}\})=\cN(H,K)$ where $K$ is either a complete multi-partite $n$-vertex graph with at most $r$ parts and a part of order at least $n-s$, or consists of a complete multipartite graph on at most $2s+1$ vertices with at most $r$ parts and possibly some isolated vertices.
\end{thm}

Note that isolated vertices may come into picture where $\ex(n,H,M_{s+1})=O(1)$, this is the case for example if $H=K_{s+1}$ by the result of Wang \cite{wang} mentioned earlier. If $r+1\ge 2s+2$, then forbidding $K_r$ does not change anything and the extremal graph is $K_{2s+1}$ with $n-2s-1$ isolated vertices. 

\begin{proof}
    First we apply Proposition \ref{propi} and we show that there is an extremal graph $G$ where $A_1$ is also complete multipartite. To do so, we apply Zykov symmetrization inside $A_1$. Obviously, after any symmetrization steps inside $A_1$, there are still at most $|B|$ independent edges incident to vertices in $B$ and at most $\lfloor |A_1|/2\rfloor$ independent edges inside $A_1$, thus we do not create $M_{s+1}$. The process terminates  as in Proposition \ref{propi}: we can pick the order of the symmetrization steps such that the number of different neighborhoods decreases after some steps.

    Assume first that there is an edge $uv$ in $A_1$. Observe that any copy of $H$ containing this edge is inside $A_1\cup B$. Indeed, a vertex $w$ outside $A_1\cup B$ is not adjacent to $u$, thus $w$ has to be in the part of $u$ in the copy of $H$. Similarly $w$ has to be in the part of $v$, but then $uv$ cannot be an edge in the copy of $H$, a contradiction. 

    Let $G'$ be the graph we obtain by deleting all the edges inside $A_1$, then $\cN(H,G)= \cN(H,G')+O(1)$. Let us delete now all the edges incident to $u$ and $v$, connect $u$ to the neighbors of an arbitrary vertex $w\not\in A_1\cup B$, and then connect $v$ to each neighbor if a vertex $w'\in B$. The set $B\cup\{v\}$ cuts the resulting graph $G''$ to parts $A_1,\dots,A_m,\{u\}$ of odd order, which satisfy the properties in the Berge-Tutte theorem, thus $G''$ is $M_{s+1}$-free. If there is a copy of $K_{r+1}$ in $G''$, it contains either $u$ or $v$, since each edge of $G''$ that is not present in $G$ is incident to at least one of these two vertices. If $u$ is contained in this copy, the rest of the vertices are in $B$, and $u$ could be replaced by $w$ to obtain a copy of $K_{r+1}$ in $G$. If $v$ is contained in that copy, then it similarly can be replaced by $w'$. The same argument works if both $u$ and $v$ are in the copy.

    We obtained that $G''$ is an $n$-vertex $\{F,M_{s+1}\}$-free graph. Assume first that each vertex of $G$ is contained in a copy of $H$. Then there is a vertex of $B$ that is contained in $\Omega(n)$ copies of $H$, we pick that as $w'$, hence $\cN(H,G'')=\cN(H,G')+\Omega(n)>\cN(H,G)$, a contradiction. 
    
    This means that if each vertex of $G$ is contained in a copy of $H$, then there are no edges in $A_1$. If $B$ has less than $r$ parts, then we can add all the edges between $B$ and $V(G)\setminus B$ without creating a forbidden graph. The resulting graph is complete multipartite with a part of order at least $n-s$, as requested. If $B$ has $r$ parts, then for each vertex $v$ outside $B$, there is a part such that $v$ is not adjacent to any vertex of that part. Then we can join $v$ to each vertex of $B$ outside that part. Repeating this for each vertex outside $B$, we obtain a complete $r$-partite graph. Proposition \ref{propp} shows that either there is a part of order at least $n-s$, or there are at most $2s+1$ vertices.

    Assume now that a vertex $v$ of $G$ is not contained in any copy of $H$. Then we can delete all the edges incident to $v$ to obtain $G_0$. If there is a vertex $u$ outside $A_1\cup B$ that is contained in a copy of $H$, then we can connect $v$ to the neighbors of $u$. The resulting graph $G_0$ is clearly $K_{r+1}$-free and also $M_{s+1}$-free because $B$ still satisfies the properties of the Berge-Tutte theorem. On the other hand $\cN(H,G_0)>\cN(H,G)$, a contradiction. We obtained that no vertex outside $A_1\cup B$ is in any copy of $H$, thus we can delete the edges that are not inside $A_1\cup B$ to obtain $G_1$ such that $\cN(H,G_1)=\cN(H,G)$.

    Now we claim that $G_1$ has chromatic number at most $r$. This follows from e.g., the strong perfect graph theorem \cite{crst}. It states that a graph $G$ has chromatic number equal to the order of the largest clique if and only if $G$ contains neither an induced odd cycle of length at least 5, nor the component of such a cycle. 

Assume that $G_1$ contains such a subgraph $C$. Then $C$ does not contain vertices with the same neighborhood, thus contains at most one vertex from each blown-up class. Therefore, $C$ contains a clique from $B$ and from $A_1$. If $C$ is a cycle, then the cliques are of order at most 2, while if $C$ is the complement of a cycle, then the cliques are of order at most $(|V(C)|-1)/2$. In both cases, there are less than $|V(C)|$ vertices in $C$, a contradiction. We remark that with a little bit more effort one could show that $G$ has chromatic number at most $r$ as well, but it is not needed for us.

Consider a proper coloring of $G_1[A_1\cup B]$ with the fewest possible colors and add all the edges between vertices in different color classes, and then add back the vertices outside $A_1\cup B$, to obtain $G_2$. Obviously we did not add any edge inside $B$, thus the resulting graph is still $M_{s+1}$-free and $K_{r+1}$-free. The number of copies of $H$ did not decrease, thus $G_2$ contains $\ex(n,H,\{M_{s+1},K_{r+1}\})$ copies of $H$. If $|A_1\cup B|\le 2s+1$, then $G_2$ satisfies the properties listed in the theorem and we are done. Otherwise there is a part of order at least $|A\cup B|-s$, in which case we can add the isolated vertices to that part without creating a forbidden graph. Then the resulting graph satisfies the properties listed in the theorem.
\end{proof}

\bigskip

\textbf{Funding}: Research supported by the National Research, Development and Innovation Office - NKFIH under the grants KH 130371, SNN 129364, FK 132060, and KKP-133819.

\end{document}